\documentclass[12pt,a4paper,oneside]{amsart}
\usepackage{amsmath}
\usepackage{xcolor}
\usepackage{hyperref}
\hypersetup{
	colorlinks   = true, 
	urlcolor     = blue, 
	linkcolor    = blue, 
	citecolor    = red 
}

\newcommand{\Href}[2]{\texorpdfstring{\hyperref[#2]{#1~\ref{#2}}}{#1~\ref{#2}}}
\usepackage{anysize}
\marginsize{2cm}{2cm}{1.5cm}{1.5cm}
\newtheorem{thm}{Theorem}
\newtheorem{prp}{Proposition}[section]
\newtheorem{lem}{Lemma}[section]
\newtheorem{cor}{Corollary}[section]
\newtheorem{conj}{Conjecture}[section]

\theoremstyle{definition}
\newtheorem{dfn}{Definition}[section]
\newtheorem{rem}[dfn]{Remark}

\newcommand{\st}{:\;}

\def\R{{\mathbb R}}%
\renewcommand{\Re}{\R}
\newcommand{\Red}{\Re^d}
\newcommand{\Redp}{\Re^{d+1}}
\renewcommand{\phi}{\varphi}

\newcommand{\ball}[1]{\mathbf{B}^{#1}}
\newcommand{\Se}{\mathbb{S}}
\newcommand{\FF}{\mathcal{F}}
\newcommand{\Sed}{\Se^{d-1}}

\providecommand{\parenth}[1]{\left(#1\right)}%
\providecommand{\braces}[1]{\left\{#1\right\}}%
\newcommand{\iprod}[2]{\left\langle#1,#2\right\rangle}%
\def\polar{\circ}
\newcommand{\polarset}[1]{{#1}^{\polar}}%
\newcommand{\id}{\mathrm{Id}}
\newcommand{\minf}[2]{\mathop\bigwedge\limits_{#1}{#2}}%

\newcommand{\conv}{\mathrm{conv}}%

\newcommand{\width}[2][K]{\mathrm{width}_{#1}\ \!\!\!\left(#2\right)}%
\newcommand{\diam}[1]{\mathrm{diam}\ \!\!\!\left(#1\right)}%
\newcommand{\norm}[1]{\left\|#1\right\|}%

\newcommand{\enorm}[1]{\left|#1\right|}
\providecommand{\parenth}[1]{\left(#1\right)}%
\providecommand{\braces}[1]{\left\{#1\right\}}%
\newcommand{\vol}[1]{\operatorname{vol}\nolimits_{#1}}%

\newcommand{\funpos}[1]{\mathcal{E}\! \left[ #1\right]}

\title{Helly numbers for Quantitative Helly-type results}
\author{Grigory Ivanov\address{Grigory Ivanov:
Pontifícia Universidade Cat\'olica do Rio de Janeiro \\
Departamento de Matematica,
Rua Marquês de São Vicente, 225\\
Edif{\'i}cio Cardeal Leme, sala 862,
22451-900 G{\'a}vea, Rio de Janeiro, Brazil}
\email{grimivanov@gmail.com}
\and
M\'arton Nasz\'odi\address{M\'arton Nasz\'odi: HUN-REN Alfr\'ed R\'enyi Inst. of
Math.; Lor\'and E\"otv\"os University, Budapest}
\email{marton.naszodi@renyi.hu}
}

\newcommand{\widthbound}{(2d)^{-2d}}
\newcommand{\widthboundreciprocal}{(2d)^{2d}}
\begin{document}
\begin{abstract}
We obtain three Helly-type results. First, we establish a Quantitative Colorful Helly-type theorem with the optimal Helly number \(2d\) concerning the diameter of the intersection of a family of convex bodies. Second, we prove a Quantitative Helly-type theorem with the optimal Helly number \(2d+1\) for the pointwise minimum of logarithmically concave functions. Finally, we present a colorful version of the latter result with Helly number (number of color classes) \(3d+1\); however, we have no reason to believe that this bound is sharp.
\end{abstract}

\maketitle
\section{Introduction}

Helly's theorem \cite{helly1923mengen}, a fundamental achievement in convexity, posits that within a finite family of convex sets in $\mathbb{R}^d$, if the intersection of any subfamily of at most $d+1$ sets shares a common point, then all the sets in the family share at least one common point. Several extensions and generalizations have been found (see \cite{HolmsenHandbook17, barany2022helly} for recent surveys). The following one, due to Lov\'asz (cf.\ \cite{barany1982generalization}), is particularly noteworthy.

\begin{prp}[Colorful Helly Theorem]\label{prp:colorful_Helly}
    Let $\mathcal{F}_1, \dots, \mathcal{F}_{d+1}$ be finite families of convex sets in $\mathbb{R}^d$. Assume that for any choice $F_1 \in \mathcal{F}_1, \dots, F_{d+1} \in \mathcal{F}_{d+1}$, the intersection $\bigcap\limits_{i=1}^{d+1} F_i$ is non-empty. Then for some $i \in \{1, \dots, d+1\}$, the intersection of all sets in the family $\mathcal{F}_i$ is non-empty.
\end{prp}

Clearly, the Helly number (the number of families -- color classes) $d+1$ is optimal in the Colorful Helly theorem. 

Two other variants of Helly’s theorem were introduced by B\'ar\'any, Katchalski, and Pach \cite{barany1982quantitative}. Their Quantitative Diameter and Volume Theorems state the following:

\begin{prp}[Quantitative Volume Theorem]
 Assume that the intersection of any $2d$ or fewer members of a finite family of convex sets in $\mathbb{R}^d$ has volume at least 1. Then the volume of the intersection of all members of the family is at least $c_d$, a strictly positive constant depending only on $d$.
\end{prp}

The current best bound on $c_d$ is due to Brazitikos \cite{brazitikos2018polynomial}, who, using the method of \cite{naszodi2016proof}, showed that one can take $c_d \approx d^{-3d/2}$.

\begin{prp}[Quantitative Diameter Theorem]\label{prp:QHD_Ambrus}
 Assume that the intersection of any $2d$ or fewer members of a finite family of convex sets in $\mathbb{R}^d$ has diameter at least 1. Then the diameter of the intersection of all members of the family is at least $\delta_d $, a strictly positive constant depending only on $d$.
\end{prp}

The current best bound on $\delta_d$ is due to Almendra-Hernández, Ambrus, and Kendall
\cite[Theorem 1.4]{almendra2022quantitative}, who, improving on ideas in \cite{ivanov2022quantitative}, showed that one can take $\delta_d = \frac{1}{2d^2}$.

To the best of our knowledge, the precise Helly numbers for colorful versions of these propositions have not been previously established.

Our first result is a Colorful Quantitative Diameter Helly theorem that involves $2d$ color classes. This number is sharp, as can be easily seen. Note that the fewer the color classes that appear in a colorful theorem of the form of Proposition~\ref{prp:colorful_Helly}, the stronger the theorem. We use 
$[n]$ to denote the set $\{1, \dots, n\}$ for a natural $n.$

\begin{thm}[Colorful Quantitative Diameter Theorem with $2d$ Colors]\label{thm:colorful_quantitative_helly}
    Let $\mathcal{F}_1, \dots, \mathcal{F}_{2d}$ be finite families of convex sets in $\mathbb{R}^d$. Assume that for any choice $F_1 \in \mathcal{F}_1, \dots, F_{2d} \in \mathcal{F}_{2d}$, the diameter of the intersection $\bigcap\limits_{i \in [2d]} F_i$ is at least $1$. Then for some $i \in [2d]$, the intersection of all sets in the family $\mathcal{F}_i$ has diameter at least $\frac{1}{2d^2\widthboundreciprocal}$.
\end{thm}

Similar results were proved by Sober\'on in~\cite{soberon2016helly} (see also~\cite{DeLoera17, de2017quantitative}), but with an exponentially large number of color classes. We believe that the bound on the diameter in Theorem~\ref{thm:colorful_quantitative_helly} is unsatisfactory and should be polynomial in~$d$.

Unfortunately, our approach does not help with the volumetric version of the result. The following problem remains open:

\begin{conj}
    Let $\mathcal{F}_1, \dots, \mathcal{F}_{2d}$ be finite families of convex sets in $\mathbb{R}^d$. Assume that for any choice $F_1 \in \mathcal{F}_1, \dots, F_{2d} \in \mathcal{F}_{2d}$, the volume of the intersection $\bigcap\limits_{i \in [2d]} F_i$ is at least $1$. Then for some $i \in [2d]$, the intersection of all sets in the family $\mathcal{F}_i$ has volume at least 
    $C_d$, for some strictly positive constant $C_d$  depending only on the dimension $d$.
\end{conj}

The authors of \cite{damasdi2021colorful} obtained such a volumetric result for $3d$ families.

We now turn to logarithmically concave functions. Recall that a function is \emph{logarithmically concave} if it is non-negative and its logarithm is a concave function. To emphasize the analogy with intersections of sets, we denote the \emph{pointwise minimum} of a set of functions $\mathcal{F}$ by $\minf{f\in\mathcal{F}}{f}$. Note that if $\mathcal{F}$ consists of log-concave functions, then $\minf{f\in\mathcal{F}}{f}$ is log-concave as well.

In \cite{ivanov2022functional}, we proved a functional version of the Quantitative Volume Theorem with Helly number $3d+2$. Our next result provides a better Helly number, $2d+1$, which is the best possible, as shown in \cite{ivanov2022functional}, but with a considerably worse bound on the integral.

\newcommand{\hellynof}{2d+1}
\begin{thm}[Functional Quantitative Helly Theorem]
\label{thm:BKP}
Let $f_1,\ldots,f_n$ be integrable log-concave functions on $\mathbb{R}^d$. 
Then there exists a subset $\sigma$ of $[n]$ of at most $\hellynof$ 
indices such 
that 
\begin{equation}\label{eq:BKPgoal}
\int_{\mathbb{R}^d} \minf{i\in\sigma}{f_i}\leq e^{C_{FQH} \cdot d^5}\int_{\mathbb{R}^d} \minf{i\in [n]}{f_i},
\end{equation}
for some absolute constant $C_{FQH}>0$.
\end{thm}

Finally, we adapt the result of \cite{damasdi2021colorful} to the functional setting by proving a colorful variant of the above result with Helly number (number of color classes) equal to $3d+1$. At this point, we believe that the number of colors could potentially be reduced to $2d+1$; it would be very interesting to see such a proof. We were unaware of any Colorful Quantitative Functional Helly-type result.

\newcommand{\chellynof}{3d+1}
\begin{thm}[Colorful Functional Helly Theorem with $3d+1$ Colors]\label{thm:Colorful_func_BKP}
Let $\mathcal{F}_1, \dots,\mathcal{F}_{\chellynof}$ be finite families of integrable log-concave functions on 
$\mathbb{R}^d$. 
Assume that for any colorful selection of $\hellynof$ functions, $f_{i_k}\in \mathcal{F}_{i_k}$ for 
each $k \in [ \hellynof]$ with $1\leq i_1<\dots<i_{\hellynof}\leq \chellynof$, the intersection function 
$\minf{k \in [ \hellynof]} f_{i_k}$ has integral greater than 1.

Then, there exists  $i \in  [\chellynof]$ such that 
\[
\int_{\mathbb{R}^d}
\left( \minf{f\in \mathcal{F}_i} f \right) \geq e^{-C_{CFQH}\cdot d^6},
\]
for some absolute constant $C_{CFQH} > 0$.
\end{thm}

In Section~\ref{sec:sets}, we prove Theorem~\ref{thm:colorful_quantitative_helly}. Then, in Section~\ref{sec:funcintro}, we introduce the background concerning logarithmically concave functions for the proofs of our other two main results, which are presented in Sections \ref{sec:funcBKP} and \ref{sec:funcColor}, respectively.

\section{Helly-type results for convex sets}\label{sec:sets}

\subsection{General notation}
We denote the standard inner product of two vectors $x,y$ in $\Red$ by $\iprod{x}{y}$, and the
Euclidean unit ball by $\ball{d}=\{x\in\Red\st \enorm{x}\leq 1\}$.
The boundary of $\ball{d}$, the unit sphere in $\Red$ is denoted as $\Sed$.
The polar of a set $X\subset\Red$ is defined and denoted as $\polarset{X}=\braces{y\in\Red\st \iprod{x}{y}\leq 1 \text{ for all } x\in X}$.
We will call a convex compact set with non-empty interior in $\R^d$ a \emph{convex body}.

We will think of $\Red$ as the linear subspace of $\Redp$ spanned by the first $d$ elements of the standard basis.

\subsection{Stirps and zones}

A \emph{strip} is a set of the form 
$S=\{x \in \Red \st  \alpha \leq \iprod{x}{u} \leq \beta\},$
where $\alpha < \beta$ and $u$ is a non-zero vector, which we call the \emph{normal} of $S$.
We say that $S$ is a \emph{supporting strip} of a set $K$ if both boundary hyperplanes of $S$ are supporting hyperplanes of $K$.
The \emph{width of a set $K$ in direction $u$}, denoted by $\width{u},$ is the distance between the boundary hyperplanes of the supporting strip of $K$ with normal $u$.
We will call the hyperplane $\{x \in \Red \st \iprod{x}{u} = \frac{\alpha + \beta}{2}\}$ the 
\emph{mid-plane} of the strip $S$.

For a  set $K \subset \Red$ and strictly positive $\lambda,$ we define 
the \emph{witness set} $A_K^{\lambda}$ by 
\[
A_K^{\lambda} = \left\{ u \in \Se^{d-1} \st  \width{u} \leq \lambda \right\}.
\]

A \emph{zone} is the intersection of an origin symmetric strip and $\Se^{d-1}$, that is, a set of the form
\begin{equation}\label{eq:zonedef}
Z=\{u\in\Sed\st -\omega\leq\iprod{u}{x}\leq \omega\},
\end{equation}
where $\omega\in [0,1]$ is the \emph{half-width} of the zone $Z$. We say that $Z$
is \emph{parallel} to any hyperplane with normal vector $u$.

We will use the following obvious estimate relating the half-width and the measure of the zone $Z$ defined by \eqref{eq:zonedef}.

\begin{equation}\label{eq:zonevol}
\sigma(Z) \leq {\omega},  
\end{equation}
where $\sigma$ is the standard Haar probability measure on the sphere $\Sed$, a notation that we will use throughout the present note. We note that bounds better than \eqref{eq:zonevol} are known (cf. \cite{boroczky2003covering}), but when $\omega$ is close to 0, then it is sufficiently accurate for our purposes.

\subsection{Bound on the diameter via witness sets}

We leave the proof of the following observation to the reader as an exercise
that relies on basic separation properties of convex sets.
\begin{lem}\label{lem:zone_inradius_bound}
Let $A$ be a centrally symmetric subset of $\Sed$ that cannot be covered by a zone of half-width $\omega$.  
Then $\conv{A}$ contains the ball $\omega\ball{d}$.
\end{lem}

The key ingredient of our proof of 
\Href{Theorem}{thm:colorful_quantitative_helly} will be the following lemma stating that if a witness set has large width, then the diameter of $K$ is bounded from above. 
\begin{lem}
\label{lem:diam_via_width_of_witness_set}
Let $K$ be a convex set in $\Red$ such that for some $\lambda>0$, the witness set  $
A_K^\lambda$ cannot be covered by a zone of half-width less than $\omega$. 
Then \[\diam{K} \leq \frac{\lambda}{\omega}.\]
\end{lem}
\begin{proof}
Consider the convex hull $\conv{A_K^\lambda}$. It is centrally symmetric and hence, its minimal half-width is at least $\omega$ by assumption.
By  \Href{Lemma}{lem:zone_inradius_bound}, 
the inradius $r$ of $\conv{A_K^\lambda}$ is at least $\omega$. 
Since the diameter of a set is the maximum of its widths, it is sufficient to show that $\width{u} \leq \frac{ \lambda}{\omega}$ for an arbitrary unit vector $u\in\Sed$.

Carath\'eodory's lemma yields that there are $u_1, \dots, u_d \in A_K^\lambda$  such that
$r u \in \conv\{0, u_1, \linebreak[0]\dots, u_d\}$.  Denote by $Q$ the polar of the set $\braces{\pm u_1, \dots, \pm u_d}$. By construction, $\width[Q]{u} \leq \frac{2}{r}.$
On the other hand, we show that a translate of $K$ is contained in
$\frac{\lambda}{2} Q$. Indeed, let $o$ be a point in the intersections of the mid-planes of $K$ with normals $u_1, \dots, u_d$. Then, by the definition of the witness set,  $K - o$ belongs to the intersection of the centrally symmetric strips of width $\lambda$ with normals $u_1, \dots, u_d$. That is,
$K-o \subset \frac{\lambda}{2} Q.$ Thus, $\width{u} \leq \frac{\lambda}{2}\width[Q]{u}\leq\frac{\lambda}{ r}\leq\frac{\lambda}{\omega}$, completing the proof of \Href{Lemma}{lem:diam_via_width_of_witness_set}.
\end{proof}


\subsection{Proof of \Href{Theorem}{thm:colorful_quantitative_helly}}
Instead of \Href{Theorem}{thm:colorful_quantitative_helly}, we will prove the following equivalent statement.
\begin{thm}\label{thm:colorful_quantitative_helly_big_big}
    Let $\FF_1, \dots, \FF_{2d}$ be finite families of convex sets in $\Red$ such that for every
    choice $F_1 \in \FF_1, \dots, F_{2d} \in \FF_{2d}$, the diameter of the intersection $\bigcap\limits_{i \in [2d]} F_i$ is greater than $\widthboundreciprocal$. 
        Then there is an $i \in [2d]$, such that the intersection of all sets in the family $\FF_i$ is of diameter greater than $\frac{1}{2d^2}$.
\end{thm}

We will use the following observation \cite[Theorem 2.2]{soberon2016helly} (see also \cite[Theorem 1.7]{de2017quantitative}).
\begin{prp}\label{prop:soberon2012helly}
Let $ \FF_1^\prime, \FF_2^\prime, \ldots, \FF_{2d}^\prime $ be non-empty finite families of convex sets in
$ \Red $, considered as color classes.
If the width in a direction $u\in\Sed$ of the intersection of every colourful choice $ F_1 \in \FF_1^\prime, \ldots, F_{2d} \in \FF_{2d}^\prime $ is at least one, then there is $i \in [2d]$ such that the width in direction $u$ of  the intersection of all sets in the family $\FF_i$ is at least one.
\end{prp}

\begin{proof}[Proof of \Href{Theorem}{thm:colorful_quantitative_helly_big_big}]
For a family of sets $\FF,$ we will use $\cap \FF$ to denote the \emph{intersection set} of $\FF,$ that is, $\cap \FF$ is the intersection of all the sets of $\FF.$
Suppose for a contradiction that the diameter of the intersection of each family is  strictly less than $\frac{1}{2d^2}$. By the version of \Href{Proposition}{prp:QHD_Ambrus} with $\delta_d = \frac{1}{2d^2}$ obtained in \cite[Theorem 1.4]{almendra2022quantitative}, every family $\FF_i$ contains a subfamily
$\FF^\prime_i\subseteq\FF_i$ of size at most $2d$ with $\diam{\cap \FF_i^\prime} < 1$ for all $i\in[2d]$.

Set $\lambda= \max\limits_{i \in [2d]}\diam{\cap \FF_i^\prime}$. Since the diameter of a set is the maximum of its widths in all directions $u\in\Sed$, it follows from \Href{Proposition}{prop:soberon2012helly} that the union of the witness sets $A_{\cap \mathcal{R}}^{\lambda}$ taken over all colorful selection $\mathcal{R}$ from $\{\FF^\prime_i\st i\in[2d]\}$ cover the sphere $\Sed$. Thus, the probability measure, $\sigma$ of one witness set, say $A_{\cap \mathcal{R}_0}^{\lambda}$, is at least $(2d)^{-2d}$, since the number of rainbow selections is $(2d)^{2d}$.

Thus, $A_{\cap \mathcal{R}_0}^{\lambda}$ cannot be covered by a zone whose measure is less than $(2d)^{-2d}$. It follows from \eqref{eq:zonevol} that the half-width of any zone covering $A_{\cap \mathcal{R}_0}^{\lambda}$ is $\omega \geq \widthbound$.

By \Href{Lemma}{lem:diam_via_width_of_witness_set}, we conclude that
\[
\diam{
\cap \mathcal{R}_0} <  \frac{1}{\widthbound},
\]
which contradicts the initial assumption. The proof of \Href{Theorem}{thm:colorful_quantitative_helly_big_big} is complete.
\end{proof}

\begin{rem}
Our goal was to obtain a Colorful Quantitative Diameter Helly-type result with the optimal Helly number. Now, after achieving such a result, one might ask for the optimal quantitative bound. We believe that more can be extracted from our approach, as we completely disregard the underlying geometry of the witness sets. For example, simple convexity arguments show that:

\begin{itemize}
    \item If $A_K^{1}$ contains a ``spherical cap'' of spherical radius $2\varphi \in \left(0, \frac{\pi}{2}\right)$, then the diameter of $K$ is at most $\frac{1}{\sin \varphi}$.
    \item If $A_K^{1}$ contains a set $D$ contained in a ``spherical cap'' of spherical radius $\varphi \in \left(0, \frac{\pi}{2}\right)$, then for any direction $u$ from the ``spherical convex hull'' of $D$, we have $\operatorname{width}_u(K) \leq \frac{1}{\cos \varphi}$.
\end{itemize}

\end{rem}

\section{Some fundamental notions from the theory of log-concave functions}\label{sec:funcintro}

Recall that the \emph{John ellipsoid} of a convex body $K$ is the largest volume ellipsoid contained in $K$. The John ellipsoid exists and is unique for any convex body. John  in his seminal paper \cite{John} (cf. K. Ball's \cite{ball1992ellipsoids}) obtained the following characterization.

\begin{prp}\label{prp:johncond}
Let $K$ be a convex body in  $\Red$ with
 $\ball{d}\subset K$.  Then the following assertions are equivalent:
\begin{enumerate}
 \item
The ball $\ball{d}$ is the John ellipsoid of $K.$
 \item
There are  points 
${u}_1,\ldots,{u}_m$ from the intersection of the boundaries of $\ball{d}$ and $K,$ and
positive weights $c_1,\ldots,c_m$ such that 
\[
\sum\limits_{i=1}^m c_i {u}_i\otimes {u}_i = \id_d 
\quad \text{and} \quad 
\sum\limits_{i=1}^m c_i {u}_i = 0.
\]
\end{enumerate}
\end{prp}

A somewhat dual construction to the John ellipsoid is the so-called 
\emph{L\"owner ellipsoid}  of a convex body $K,$ which is the smallest volume ellipsoid containing $K.$ Similar to \Href{Proposition}{prp:johncond},
one gets
\begin{prp}\label{prp:lownercond}
Let $K$ be a convex body in  $\Red$ with
 $\ball{d}\supset K$.  Then the following assertions are equivalent:
\begin{enumerate}
 \item
The ball $\ball{d}$ is the L\"owner ellipsoid of $K.$
 \item
There are  points 
${u}_1,\ldots,{u}_m$ from the intersection of the boundaries of $\ball{d}$ and $K,$ and
positive weights $c_1,\ldots,c_m$ such that 
\[
\sum\limits_{i=1}^m c_i {u}_i\otimes {u}_i = \id_d
\quad \text{and} \quad 
\sum\limits_{i=1}^m c_i {u}_i = 0.
\]
\end{enumerate}
\end{prp}
A standard corollary of  \Href{Proposition}{prp:lownercond} is as follows:
\begin{prp}\label{prp:lowner_inclusion}
Let the ball $\ball{d}$ be the L\"owner ellipsoid of a convex body  $K$ in $\R^d.$
Then $K \supset \frac{\ball{d}}{d}.$
\end{prp}


Recall that a function $f$ on $\Red$ is \emph{logarithmically concave}
(or in short, \emph{log-concave}), if  it takes non-negative values and
its logarithm is a concave function on $\Red$, that is, when $f=e^{-\psi}$ for a convex function $\psi:\Red\longrightarrow \Re\cup\{+\infty\}$.
We will call an upper semi-continuous  function of positive and finite integral a \emph{proper} function.

As a natural generalization of the notion of 
affine images of convex bodies, we define the \emph{positions} of a function $g$ 
on $\Red$ as
\[\funpos{g} = \{\alpha g(Ax+a)\st A\in\Re^{d\times d} \text{ non-singular}, \alpha>0, a\in\Red\}.\] We call the function $g(x+a)$ for a fixed vector $a\in\Red$ a \emph{translate} of $g$ by  $a$.
We will say that a function $f_1$ on $\Red$ is \emph{below} another function $f_2$ on $\Red$ (or that $f_2$ is \emph{above} $f_1$)
and denote it as $f_1 \leq f_2,$ if $f_1$ is pointwise less than or equal to $f_2,$
that is, $f_1(x) \leq f_2(x)$ for all $x \in \Red.$  

\newcommand{\hballf}{\hbar}
We define the \emph{height function} of $\ball{d+1}$ as
\begin{equation}\label{eq:hdef}
\hballf (x) =
 \begin{cases}
\sqrt{1 - \enorm{x}^2},&
\text{ if }x \in \ball{d}\\
0,&\text{ otherwise}.
 \end{cases}
\end{equation}
Clearly,
\begin{equation}
\label{eq:heightf_integral}
\int_{\Red} \hballf =  \frac{1}{2}\vol{d+1}\ball{d+1} \geq 
 \frac{\vol{d}\ball{d}}{2d}.
 \end{equation}
We will refer to a function $g$ from the set $\funpos{\hballf}$ as an \emph{ellipsoidal function}. 

The \emph{support} of a log-concave function $f$ is the set $\braces{x\in\Red\st f(x)>0}$ in $\Red$.
The closure of the support of an ellipsoidal function $g$ is an ellipsoid in $\Red$, which we will call the \emph{base ellipsoid} \label{def:base} of $g$.

We will say that $\tilde{g}$ is the \emph{John function} of $f$ if it is a solution to the following problem.
 
 \medskip \noindent
\textbf{The John problem:} Find
\begin{equation}\label{eq:john_problem_intro}
\max\limits_{g \in \funpos{\hballf} } 
	\int_{\Red} g
	\quad \text{subject to} \quad
	g \leq f.
\end{equation}

In \cite{ivanov2022functional}, the authors proved that  John problem \eqref{eq:john_problem_intro} has a unique solution for any proper log-concave function $f
$ and the following functional analogue of \Href{Proposition}{prp:johncond} holds (\cite[Theorem~5.1]{ivanov2022functional}):
 
\begin{prp}\label{prp:johncondfunc}
Let $f$ be a proper  log-concave function on $\Red$ such that 
 $\hballf\leq f$.  Then the following assertions are equivalent:
\begin{enumerate}
 \item
The function $\hballf$ defined in \eqref{eq:hdef} is the John function of $f.$
 \item
There are  points 
${u}_1,\ldots,{u}_m \in \ball{d}\subset \Red,$ which will refer to as  
 \emph{John contact points of $f$}, 
and 
positive weights $c_1,\ldots,c_m$ such that 
\begin{enumerate}
\item $f(u_i) = \hballf(u_i)$ for all $i \in [m]$;
\item $\sum\limits_{i=1}^m c_i {u}_i\otimes {u}_i = \id_d;$
\item $\sum\limits_{i=1}^m c_i f({u}_i) \cdot f({u}_i) = 1;$
\item $\sum\limits_{i=1}^m c_i u_i =0.$
\end{enumerate}
\end{enumerate}
\end{prp} 
 
We quote the following inequality on the \emph{integral ratio}, cf. \cite[Corollary~6.1]{ivanov2022functional}.
 \begin{equation}\label{eq:integral_ratio}
 \parenth{\frac{\int_{\Red} f}{\int_{\Red} g}}^{1/d} \leq \Theta \sqrt{d},
 \end{equation}
 where $g$ is the John function of $f$ and $\Theta$ is an absolute constant.
 
We quote \cite[Lemma~3.1]{ivanov2022functional}, a straight-forward observation that follows from the definitions, and provides a simple way of bounding a log-concave function $f$ from above by a function whose logarithm is a linear functional.
\begin{lem}[Upper bound by a log-linear function]\label{lem:touching_cond_log_concave}
Let $\psi_1$ and $\psi_2$ be  convex functions on $\Red$ and $f_1= e^{-\psi_1}$ 
and $f_2 = e^{-\psi_2}.$ 
Let $f_2\leq f_1$ and $f_1(x_0) = f_2(x_0) >0$ at some point $x_0$ 
in 
the interior of the domain of $\psi_2.$ 
Assume that $\psi_2$ is differentiable at $x_0$. Then $f_1$ and $f_2$ 
are differentiable at $x_0$, and $\nabla f_1(x_0)= \nabla f_2(x_0)$ moreover, 
\[
f_1(x) \leq f_2(x_0)e^{- \iprod{\nabla\psi_2(x_0)}{x-x_0}}
\]
for all $x \in \Red.$
\end{lem}
For every  $u \in \ball{d} \subset \Red,$ define a function $\ell_{u} \colon \Red \to [0, \infty]$ by
\[
\ell_{u}(x) =
\begin{cases} 
  \hballf(u) e^{- \frac{1}{\hballf^2(u)}\iprod{u}{x-u}}, & \text{if } \enorm{u} < 1 \\
  0, & \text{if } \enorm{u} = 1, \text{ and } \iprod{x}{u} \geq 1 \\
  +\infty, & \text{if } \enorm{u} = 1, \text{ and }\iprod{x}{u} < 1
\end{cases}.
\]

Applying \Href{Lemma}{lem:touching_cond_log_concave} for any log-concave function $f_1=f$ and $f_2=\hballf$, and using our notation $\ell_u$, we obtain \cite[Corollary~3.1]{ivanov2022functional}:
\begin{cor}\label{cor:touchingballbound}
 Let $f$ be a log-concave function on $\Red$ such that $f \geq \hballf$ and for some unit vector $(u, \hballf(u))$ with $u \in \Red$  such that $f(u) = \hballf(u)$. Then  $f \leq \ell_u.$
\end{cor}

\section{Proof of \Href{Theorem}{thm:BKP}}\label{sec:funcBKP}

\Href{Corollary}{cor:touchingballbound} can be applied to one contact point of $f$ and $\hballf$.
Our goal is to extend it to multiple contact points.

Recall that the \emph{gauge function} of a convex body $K \subset \R^d$ containing the origin in its interior is denoted by $\norm{\cdot}_K.$  
The vertex set of  a polytope $P$ is denoted by $\operatorname{vert} P.$ 
\begin{lem}[A tail bound for the minimum of log-linear functions]\label{lem:contact_function_via_norm}
Let $P$ be a polytope in $\Red$ with
$\delta \ball{d} \subset P \subset \ball{d}$ for some $\delta\in(0,1)$.
Define  $\tilde{g} \colon \Red \to [0, \infty]$ by 
$
\tilde{g} = \minf{u \in \operatorname{vert} P}{ \ell_{u}}.
$
Then
\[
\tilde{g}(x) \leq  e \cdot \exp\parenth{-\norm{x}_{P^{\circ}}}
\]
for all $x \in \Red \setminus P^\circ$.
\end{lem}
\begin{proof}
Clearly, $\tilde{g}$ is a log-concave function, and
$
 \norm{x}_{P^\circ} = \max\braces{\iprod{u}{x}
   \st  u\in \operatorname{vert}{P}}.
$
For a given $x \in \Red \setminus P^\circ,$ let 
$u\in \operatorname{vert}{P}$
satisfy the identity $ \norm{x}_{P^\circ} = \iprod{u}{x}$. By the choice of $x$ and $u,$
$\iprod{u}{x} \geq 1 \geq \iprod{u}{u}.$ If $\enorm{u} = 1$, then $\tilde{g}(x) = 0$, and the desired inequality trivially holds.

Consider the case $0 < \enorm{u} < 1.$ Then
\[
\tilde{g}(x) \leq  \ell_u(x) = \hballf(u) e^{- \frac{1}{\hballf^2(u)}\iprod{u}{x-u}} \leq e^{- \frac{1}{\hballf^2(u)}\iprod{u}{x-u}}  \leq e^{- \iprod{u}{x-u}}  \leq e \cdot e^{- \iprod{u}{x}} = e \cdot \exp\parenth{-\norm{x}_{P^{\circ}}},
\]
where the last two inequalities follow from the facts that $\hballf(u) \in (0,1]$ and $\iprod{u}{x- u} \geq 0$. The proof of \Href{Lemma}{lem:contact_function_via_norm} is complete.
\end{proof}
\begin{cor}[Contact points yield a tail bound]\label{cor:tailsfunctionalhelly}
Assume that $P \subset \Red$ is a polytope satisfying the inclusion
$\delta \ball{d} \subset P \subset \ball{d}$ for some $\delta\in(0,1)$.
 Let $g$ be a log-concave function on $\Red$ such that $g \geq \hballf$ and $g(u) = \hballf(u)$, for every vertex $u$ of $P$. Then
\[
\int_{\Red\setminus P^{\circ}} 
g \leq 2e  \cdot \frac{d^{d+1}}{\delta^d} \cdot \int_{\Red} \hballf.
\]
\end{cor}

\begin{proof}
Define  $\tilde{g} \colon \Red \to [0, \infty)$ by 
$
\tilde{g} = \minf{u \in \operatorname{vert} P}{ \ell_{u}}.
$
By \Href{Corollary}{cor:touchingballbound}, $g \leq \tilde{g}$. Using \Href{Lemma}{lem:contact_function_via_norm}, one obtains
\[
\int_{\Red\setminus P^{\circ}} g \leq \int_{\Red\setminus P^{\circ}} \tilde{g} \leq
 e\int_{\Red\setminus P^{\circ}} 
  \exp\left(-\norm{x}_{P^{\circ}}\right)\leq
 e\int_{\Red} 
  \exp\left(-\norm{x}_P^{\circ}\right)=
e\cdot d!\vol{d}P^{\circ}
\leq 
\]\[
e \cdot d^{d}\vol{d}P^{\circ}\leq 
e \cdot \frac{d^{d}}{\delta^d} \vol{d}\ball{d} \leq  
e \cdot \frac{d^{d}}{\delta^d} \cdot d \vol{d+1}\ball{d+1} =
{2ed}  \cdot \frac{d^{d}}{\delta^d} \cdot \int_{\Red} \hballf.
\]
\end{proof}

We will use \cite[Theorem~1]{ivanov2024steinitz}.
\begin{prp}[Quantitative Steinitz Theorem \cite{ivanov2024steinitz}]\label{prp:IN_sparsification}
Let $Q$ be a convex polytope in $\Red$ containing the ball $\ball{d}$. Then there exists a subset of at most $2d$ vertices of $Q$ whose convex hull $\tilde{Q}$ satisfies
\[
\frac{1}{6d^2}\ball{d}\subseteq\tilde{Q}.
\]
\end{prp}

\begin{proof}[Proof of \Href{Theorem}{thm:BKP}]
Without loss of generality, we assume that $\hballf$ is the John function of  $\minf{i\in[n]}{f_i}$.
Let $u_1, \dots, u_m \in \Red$ be the John contact points of $\minf{i\in[n]}{f_i}$ as defined in \Href{Proposition}{prp:johncondfunc}.
Define $Q$ as their convex hull, and define $\overline{Q}$ by
\[
\overline{Q}= \conv\braces{(u_i, \pm \hballf(u_i))},
\]
a convex polytope in $\Redp$ symmetric about $\Red$.
It follows from \Href{Proposition}{prp:lownercond} and \Href{Proposition}{prp:johncondfunc} that $\ball{d+1}$ is the L\"owner ellipsoid of the convex polytope $\overline{Q}.$ Hence,
$\overline{Q} \supset \frac{1}{d+1} \ball{d+1}$ by \Href{Proposition}{prp:lowner_inclusion}. Since $Q$ is the orthogonal projection of  $\overline{Q}$, we have ${Q} \supset \frac{1}{d+1} \ball{d}$. Thus, \Href{Proposition}{prp:IN_sparsification}, yields that there is $\tau_1 \subset [m]$ of size at most $2d$, such that
the polytope $P$ defined by $P = \conv\braces{u_j \st j \in \tau_1}$ satisfies the inclusion
\[
\frac{1}{12 d^3} \ball{d} \subset P \subset \ball{d}.
\]
For each $j\in[m]$, we pick an $i(j)\in[n]$ with $f_{i(j)}(u_j)=\minf{i\in[n]}{f_i}(u_j)$, and set
$\sigma_1=\{i(j)\st j\in\tau_1\}$. Hence, by \Href{Corollary}{cor:tailsfunctionalhelly},
\begin{equation}
\label{eq:tails_bound}
\int_{\Red\setminus P^{\circ}} 
\minf{i\in\sigma_1}{f_i} \leq 2ed  \cdot 12^d {d^{4d}} \cdot \int_{\Red} \hballf.
\end{equation}

In order to complete the proof, we need to find one more index $j\in[m]$ such that the integral of $f_{i(j)}$ is bounded from above on $P^{\circ}$.

Let $j\in[m]$ be such that among the vertices of $Q$, the vertex $u_j$ is of minimal Euclidean norm.
We claim that $\sigma=\sigma_1\cup\{i(j)\}$ satisfies the conclusion of the theorem.
Since $u_j$ is a vertex of $P$, we have $\hballf^2(u_j) = 1 - \enorm{u_j}^2 \geq \frac{1}{(d+1)^2} \geq \frac{1}{4d^2}$, which helps bound $f_j$ on $P^{\circ}$ as follows. 
\[
\max\limits_{x \in P^{\circ}} {f_j}(x) \leq
\max\limits_{x \in P^{\circ}}{{\hballf(u_j)} e^{-\frac{1}{\hballf^2(u_j)}\iprod{u_j}{x-u_j}}} \leq 
\max\limits_{x \in P^{\circ}}{ e^{-\frac{1}{\hballf^2(u_j)}\iprod{u_j}{x-u_j}}} \leq
\]\[
e^{4 d^2 }\max\limits_{x \in P^{\circ}}{ e^{-\frac{1}{\hballf^2(u_j)}\iprod{u_j}{x}}} \leq
e^{4 d^2 }\max\limits_{x \in P^{\circ}}{ e^{4 d^2\enorm{x}}} \leq
e^{4 d^2 }\cdot{ e^{4 d^2\cdot12d^3}} \leq
e^{52 d^5}.
\]
Combined with \eqref{eq:tails_bound}, we obtain a bound on the integral of $\minf{i\in\sigma}{f_i}$.
\[
 \int_{\Red} \minf{i\in\sigma}{f_i} \leq \int_{\Red\setminus P^{\circ}} \minf{i\in\sigma_1}{f_i} +
\int_{P^{\circ}} f_{i(j)}\leq
 \int_{\Red\setminus P^{\circ}} \minf{i\in\sigma_1}{f_i} +
e^{52 d^5}\vol{d}(P^{\circ})\leq
\]\[
\int_{\Red\setminus P^{\circ}} \minf{i\in\sigma_1}{f_i} +
e^{52 d^5}\vol{d}(12d^3\ball{d})\leq
2ed  \cdot 12^d {d^{4d}} \cdot \int_{\Red} \hballf +
2\cdot e^{52 d^5}\cdot 12^d d^{3d}\int_{\Red} \hballf.
\]
Finally, using our bound \eqref{eq:integral_ratio} and the assumption that $\hballf$ is the John function of $\minf{i\in[n]}{f_i}$, we have
\[
 \int_{\Red} \minf{i\in\sigma}{f_i} \leq
\left[2ed  \cdot 12^d {d^{4d}} +
2\cdot e^{52 d^5}\cdot 12^d d^{3d}\right]
\cdot  \Theta^d d^{d/2} \int_{\Red} \minf{i\in[n]}{f_i}
\leq
e^{c d^5}\int_{\Red} \minf{i\in[n]}{f_i},
\]
for some absolute constant $c>0$, completing the proof of \Href{Theorem}{thm:BKP}.
\end{proof}

\section{Colorful functions}\label{sec:funcColor}

The goal of this section is to prove \Href{Theorem}{thm:Colorful_func_BKP}.
Our proof is a straightforward adaptation to the functional setting of the proof of Corollary 1.2 in \cite{damasdi2021colorful}.

We leave the following fact as an exercise.
\begin{lem}[Translates of a function under another]\label{lem:logconc_positions_below}
Let $f \colon \Red \to [0, \infty)$ be a log-concave function and  
$g \colon \Red \to [0, \infty)$ be a non-negative function. 
Assume that the translates $g(x+a_1)$ and $g(x+a_2)$ of $g$ by vectors $a_1, a_2\in\Red$ are below $f$.
Then $g(x+\lambda a_1+ (1-\lambda)a_2)$ is also below $f$ for any $\lambda \in [0,1]$.
\end{lem}

One of our key lemmas states that the John ellipsoidal function of a log-concave function $f$ has a shrunk copy in any ellipsoidal function below $f$ of not too small integral.
\begin{lem}[Big ellipsoids contain a small copy of the John ellipsoid]\label{lem:hellyklee}
Assume that $\hballf$ is the John function of a proper log-concave function 
$f \colon \Red \to [0, \infty)$ and a certain position $\tilde{h} \subset \funpos{\hballf}$ of $\hballf$ (that is, an ellipsoidal function) satisfies the inequalities  $\tilde{h} \leq f$ and
$\int_{\Red} \tilde{h} \geq \delta \int_{\Red} {\hballf}.$
Then $g=\delta \cdot \parenth{\frac{\delta}{4}}^d \cdot \hballf{} \circ \frac{4\id_d}{\delta}$ has a translate below $\tilde{h}$.
\end{lem}
\begin{proof}
Let the base ellipsoid of $\tilde{h}$ (see p.\pageref{def:base}. for the definition) be $A \ball{d} + a$ for some positive definite matrix $A$ and $a\in\Red$ , and let $\alpha$ denote the maximum of $\tilde{h}$.
By the log-concavity of $f$, the ellipsoidal function
$\overline{h}$ with base ellipsoid $\frac{\id_d + A}{2} \ball{d} + \frac{a}{2}$ and maximum $\sqrt{\alpha}$ is below $f$ as well.
Since   $\hballf$ is the John function of $f$ and $\int_{\Red} \tilde{h} \geq \delta \int_{\Red} {\hballf}$, we get
\[
\frac{1}{\sqrt{\delta}} \leq 
\frac{\int_{\Red} \overline{h}}{\sqrt{\int_{\Red} \tilde{h} \int_{\Red} \hballf}} =
\frac{1}{2^d} \frac{\det\parenth{\id_d + A}}{\sqrt{\det A} \sqrt{\det{\id_d}}} =
\frac{1}{2^d} \frac{\det\parenth{\id_d + A}}{\sqrt{\det A}}.
\]
Let $\beta_1, \dots, \beta_d$ be the eigenvalues of $A.$ Diagonalizing $A$ and using it in the previous inequality, one obtains
\[
\frac{1}{\sqrt{\delta}} \leq \prod\limits_{i \in [d]} \frac{1+\beta_i}{2\sqrt{\beta_i}}.
\]
Since $1 + \beta \geq 2 \sqrt{\beta}$ for any $\beta>0$, each term on the right is at least one, and hence,
\[
\frac{1}{\sqrt{\delta}} \leq \frac{1+\beta_i}{2\sqrt{\beta_i}}, \text{ for all } i\in[d].
\]
Fix an $i\in[d]$. Setting $t = \sqrt{\beta_i}$ yields
$
t^2 - \frac{2}{\sqrt{\delta}} t  +1 \leq 0.
$
By solving this quadratic inequality, one gets
$
\frac{1}{\sqrt{\delta}} + \sqrt{\frac{1}{\delta} - 1} \geq t \geq 
\frac{1}{\sqrt{\delta}} - \sqrt{\frac{1}{\delta} - 1},
$
which shows that
\begin{equation}\label{eq:betabounds}
\frac{4}{\delta} \geq \beta_i \geq \frac{\delta}{4}, \text{ for all } i\in[d].
\end{equation}

Next,
\[
 \int_{\Red}\tilde{h}= \alpha \prod\limits_{i \in [d]} \beta_i\cdot \int_{\Red}\hballf.
\]

Using the upper bound on $\beta_i$ in \eqref{eq:betabounds}, we conclude that $\alpha \geq
\delta \cdot \parenth{\frac{\delta}{4}}^d$.
The lower bound in \eqref{eq:betabounds} yields that
\[
 A\ball{d}+a \supseteq\frac{\delta}{4}\ball{d}+a.
\]
Thus, $\hballf$ and $g$ have concentric ellipsoidal bases, where the former's base contains the latter's, and the former's maximum is at least that of the latter. It follows that $\hballf$ is pointwise above $g$, completing the proof of \Href{Lemma}{lem:hellyklee}.
\end{proof}
\Href{Lemma}{lem:hellyklee} is an extension of Lemma 3.2 in \cite{damasdi2021colorful} to our functional setting. Moreover, applying \Href{Lemma}{lem:hellyklee} to the indicator function of a convex set yields Lemma 3.2 of \cite{damasdi2021colorful} with a tighter bound.

We will use the following consequence of the Colorful Helly theorem, \Href{Proposition}{prp:colorful_Helly}.
\begin{cor}\label{cor:chellyklee}
Let $\FF_1,\dots, \FF_{d+1}$ be finite families of log-concave functions on $\Red$. Assume that for any colorful selection $f_1\in 
\FF_1,\dots, f_{d+1}\in \FF_{d+1}$, the function
$\minf{i\in[d+1]}{f_i}$ is above a translate of a given function $g$.
Then for some $j$, the intersection function $\minf{f\in \FF_j} f$ is above a translate of $g.$
\end{cor}
\begin{proof}
By \Href{Lemma}{lem:logconc_positions_below}, for functions $f$ and $g$ on $\Red$, the set of vectors $a$ such that
the translate of $g$ by $a$ is below $f$ is a convex set in $\Red$ provided that $f$ is log-concave. The statement now follows from the Colorful Helly theorem (\Href{Proposition}{prp:colorful_Helly}).
\end{proof}

\subsection{An ordering on ellipsoidal functions}
Followinf Lovász’ idea of the proof of the Colorful Helly Theorem, \Href{Proposition}{prp:colorful_Helly}, the authors of \cite{damasdi2021colorful} consider a certain  ordering on ellipsoids in $\mathbb{R}^d.$ They said that an ellipsoid $E_1$ is \emph{lower than} an ellipsoid $E_2$ if the highest point of the projection of $E_1$ onto the last coordinate axis is below the highest point of the projection of $E_2$ onto the last coordinate axis.

We adopt the above ordering to an ordering of ellipsoidal functions as follows. Let $h_1$ and $h_2$ be two ellipsoidal functions with maxima $\alpha_1$ and $\alpha_2$, respectively. We say that $h_1$ is \emph{lower than} $h_2$ if $\alpha_1 < \alpha_2$.

Note that the pointwise minimum of a log-concave function and a constant function is log-concave. From the existence and uniqueness of the John function obtained in \cite{ivanov2022functional}, it follows that if the John function of a proper log-concave function $f$ has integral at least $\int_{\mathbb{R}^d} \hballf$, then there exists a unique lowest ellipsoidal function with integral $\int_{\mathbb{R}^d} \hballf$. We call it the \emph{lowest ellipsoidal function} of $f$.

\subsection{Proof of \Href{Theorem}{thm:Colorful_func_BKP}}
Note that we may assume that all the functions in all the families in the theorem are proper log-concave functions. Indeed, we define a function $\hballf_r=r\cdot\hballf\circ\frac{\id_d}{r}$ with a large $r>0$. Since we have only finitely many functions, by replacing every function $f$ in every family with the pointwise minimum of $f$ and $\hballf_r$, the assumption of the theorem will remain valid as long as $r$ is sufficiently large.

Inequality \eqref{eq:integral_ratio} states that the integral of a proper log-concave function is essentially the same (up to a negligible factor) as the integral of its John function. In particular, it is sufficient to show the following statement, where we are chasing ellipsoidal functions and not integral directly.

\begin{thm}\label{thm:Colorful_func_BKP_restated}
Let $\FF_1, \dots,\FF_{\chellynof}$ be finite families of proper log-concave functions on
$\Red$.
Assume that for any colorful selection of $\hellynof$ functions, $f_{i_k}\in \FF_{i_k}$ for
each $ k \in [\hellynof]$ with $1\leq i_1<\dots<i_{\hellynof}\leq \chellynof$, the John function of
$\minf{k \in [\hellynof]}{f_{i_k}}$ is of integral greater than $\int_{\Red}\hballf$.

Then, there exists  $ i \in  [\chellynof]$ such that
$\minf{f\in \FF_i}{f}$ is greater than an ellipsoidal function of integral at least $e^{-C_{CFQH}^\prime \cdot d^{6}}\int_{\Red}\hballf$ for some absolute constant $C_{CFQH}^\prime > 0$.
\end{thm}
\begin{proof}[Proof of \Href{Theorem}{thm:Colorful_func_BKP_restated}]
Consider the lowest ellipsoidal functions of the pointwise minimum function of all colorful selections of
$2d$ functions. 
We may assume that the highest one of these ellipsoidal functions is $\hballf$. 
By possibly changing the indices of the families, we may assume that the 
selection is $f_1\in\FF_1,\dots,f_{2d}\in\FF_{2d}$. 
We call $\FF_{\hellynof}, \dots,\FF_{\chellynof}$ the \emph{remaining families}.

Define $H_1 \colon \Red \to [0, \infty)$ by $H_1 \equiv 1$. By our choice and
the ordering of ellipsoidal functions,
$\hballf$ is the John function of
\[\tilde{f} = \minf{f\in\braces{f_1,\dots f_{2d}, H_1}}{f}.\]

Next, take an arbitrary colorful selection 
$f_{\hellynof}\in\FF_{\hellynof}, \dots,f_{\chellynof}\in\FF_{\chellynof}$ of the 
remaining $d+1$ families. We 
claim that the pointwise minimum function of any $\hellynof$ elements of the sequence
\[f_1,\dots, f_{2d}, H_1, f_{\hellynof},\dots, f_{\chellynof}\]
is pointwise above an ellipsoidal function of integral at least $\int_{\Red}\hballf$.
Indeed, if $H_1$ is not among those $\hellynof$ elements, then the assumption of \Href{Theorem}{thm:Colorful_func_BKP_restated} ensures it.
If $H_1$ is among them, then by the choice of $H_1$, the claim holds.
Therefore, by \Href{Theorem}{thm:BKP} and \eqref{eq:integral_ratio}, the function
\[
 \minf{f\in\braces{f_1,\dots,f_{\chellynof}, H_1}}{f}
\]
is pointwise greater than  an ellipsoidal function  $\tilde{h}$ of integral at least
$\delta\int_{\Red}\hballf,$ where $\delta=\exp( - C_{FQH} \cdot d^5) \cdot \Theta^d d^{d/2}$.
Since $\hballf$ is the John ellipsoidal function  of $\tilde{f}$, by
\Href{Lemma}{lem:hellyklee}, we 
conclude that  a translate of  $g=\delta \cdot \parenth{\frac{\delta}{4}}^d \cdot \hballf{} \circ \frac{4\id_d}{\delta}$ is pointwise below $\tilde{h}$, and hence, this translate of $g$ is pointwise below $\minf{i\in[\chellynof]\setminus [2d]}{f_i}$.

Thus, we have shown that for any colorful selection
$f_{2d+1}\in\FF_{2d+1}, 
\dots,f_{\chellynof}\in \FF_{\chellynof}$ of the 
remaining $d+1$ families, the minimum
$\minf{i\in[\chellynof]\setminus [2d]}{f_i}$ is pointwise above a translate of the  ellipsoidal function $g$. It follows from \Href{Corollary}{cor:chellyklee}
that there is 
an index $i\in[\chellynof]\setminus [2d]$ such that $\minf{f\in \FF_i}{f}$ is pointwise above a translate of $g$. Finally,
\[
 \int_{\Red} g \geq e^{- C_{CFQH}^\prime \cdot d^6} \cdot\int_{\Red} \hballf
\]
for some universal constant $C_{CFQH}^\prime > 0,$
completing the proof of \Href{Theorem}{thm:Colorful_func_BKP_restated}.
\end{proof}

\bibliographystyle{alpha}
\bibliography{../uvolit}
\end{document}